\documentclass[12pt]{extarticle}

\usepackage[english]{babel}
\usepackage{amssymb,amsthm,amsmath,mathtools}
\usepackage{subcaption}
\usepackage{tikz}
\usetikzlibrary{calc}
\usetikzlibrary{decorations.pathreplacing}
\tikzstyle{snode}=[circle ,draw=black,fill=red,thick, inner sep=0pt ,minimum size=1.4mm]
\tikzstyle{bnode}=[circle ,draw=black,fill=black,thick, inner sep=0pt ,minimum size=1.4mm]
\usepackage{textcomp}
\usepackage{ifthen}
\usepackage[normalem]{ulem}

\usepackage{enumitem}
\setlist[enumerate]{left=0pt,label=(\roman*),
ref=(\roman*),font=\normalfont,topsep=-1ex,parsep=-.3ex,partopsep=0pt}

\usepackage{hyperref}
\hypersetup{
    colorlinks=true,
    pdftitle={Spanning trees of smallest maximum degree in subdivisions of graphs},
    pdfauthor={Brause, Harant, Hörsch, Mohr},
    bookmarks=true,
    pdfpagemode=UseOutlines,
}
\usepackage[capitalise]{cleveref}

\usepackage[left=2.5cm,right=2.5cm,top=3cm,bottom=3cm]{geometry} 
\newtheorem{theorem}{Theorem}

\newtheorem{lemma}{Lemma}
\newtheorem{claim}{Claim}

\newtheorem{case}{Case}
\newtheorem{proposition}{Proposition}

\newenvironment{claimproof}{\noindent \textit{Proof.}}{{\hfill\tiny ($\Box$)\newline}}

\title{Minimum  Spanning Trees with Bounded Degrees of Vertices in a Specified Stable Set}

\author{Christoph Brause$^1$, Jochen Harant$^2$, Florian Hörsch$^2$, Samuel Mohr$^{3,}$\thanks{supported by the MUNI Award in Science and Humanities of the Grant
Agency of Masaryk university.}\\
\small $^1$ Institute for Discrete Mathematics and Algebra\\[-0.8ex]
\small TU Bergakademie Freiberg\\[-0.8ex] 
\small Germany\\
\small $^2$ Institute for Mathematics\\[-0.8ex] 
\small TU Ilmenau\\[-0.8ex] 
\small Germany\\
\small $^3$ Faculty of Informatics,\\[-0.8ex]
\small Masaryk University\\[-0.8ex]
\small Czech Republic
}

\begin{document}

\maketitle 

\begin{abstract}

Given a graph $G$ and sets $\{\alpha_v~|~v \in V(G)\}$ and $\{\beta_v~|~v \in V(G)\}$  of non-negative integers, it is known that the decision problem  whether $G$ contains a spanning tree $T$ 
such that $\alpha_v \le d_T (v) \le \beta_v $ for all $v \in  V(G)$ is $NP$-complete. \\In this article,  we relax the problem by  demanding  that the degree restrictions apply to vertices $v\in U$ only, where  $U$ is a stable set of $G$. In this case, the problem becomes tractable.\
 A. Frank presented a result  characterizing the positive instances of that relaxed problem. Using matroid intersection developed by J. Edmonds,  we give a new and short proof of Frank's result and show that if $U$ is  stable and  the edges of  $G$ are weighted by arbitrary real numbers, then even a minimum-cost tree $T$ with $\alpha_v \le d_T (v) \le \beta_v $ for all $v \in  U$ can be found in polynomial time if such a tree exists.

\end{abstract}

\section{Introduction and Main Results}\label{IR}

Spanning trees play a crucial role in network theory as they can be considered the smallest substructures that allow to maintain the connectivity of the network. In \cite{OY}, a survey on results on spanning trees in graphs is presented.

We consider undirected, finite,  and simple graphs $G$ with vertex set $V(G)$ and edge set $E(G)$, 
 use standard notation of graph theory, and refer the reader to \cite{RhD} for details. 
We say that a vertex set $S\subseteq V(G)$ of a graph $G$  is {\em stable} if the subgraph  of $G$ induced by $S$ is edgeless. We wish to remark that this property is often referred to as being ``independent'' in the literature. In order to avoid confusion with independent sets of a matroid, we choose the aforementioned notation.

In the Minimum Bounded Degree Spanning Tree problem \cite{SL}, we are given a  graph $G$,  a real weight function $c:E(G)\rightarrow R$, and  
sets $\{\alpha_v~|~v \in V(G)\}$ and $\{\beta_v~|~v \in V(G)\}$ of non-negative integers such that $\alpha_v \le \beta_v$ for $v\in V(G)$. We say that a spanning tree $T$ of $G$ is feasible if   $\alpha_v \le d_T (v) \le \beta_v $ holds for all $v \in V(G)$.
The task is to decide whether $G$ has a feasible spanning tree $T$ and if this is the case, to compute one of minimum weight. 
 
 It is known that already checking whether there exists a feasible spanning tree for that problem is
$NP$-complete since  a Hamiltonian path of $G$  is exactly a feasible spanning tree for the special case  $\alpha_v=0$ and $\beta_v=2$ for all $v\in V(G)$ \cite{gj}.
On the other hand,  a best possible approximation algorithm for  the Minimum Bounded Degree Spanning Tree problem is proved in \cite{SL}. 

Here we relax the problem by demanding the degree conditions $\alpha_v \le d_T (v) \le \beta_v $ only for vertices $v$ of a stable set $U$ of $G$.
The forthcoming Theorem \ref{1} presents a characterisation of the positive instances for this relaxed problem and gives hope that it could be easier.  We use $\omega(G)$ for the number of components of  $G$. If $S \subseteq V(G)$, then $G-S$ is the subgraph of $G$ obtained from $G$ by removing $S$ and $N(S)$ denotes the set of vertices of $G-S$ adjacent to a vertex of $S$ in $G$. Note that Condition (\ref{beta}) of Theorem \ref{1} in case $S=\emptyset$ implies that $G$ is connected.

\begin{theorem}[A. Frank, \cite{book}, Theorem 9.1.16] \label{1}
Let $G$ be a graph, $U$ be a  stable subset of $V(G)$, and $\{\alpha_v~|~v \in U\}$ and $\{\beta_v~|~v \in U\}$ be sets of non-negative integers such that $\alpha_v\le \beta_v$ for all $v\in U$. Then $G$ contains a spanning tree $T$ such that $\alpha_v\leq d_T(v)\leq \beta_v$ for all $v \in U$  if and only if the following two conditions hold:

 \begin{equation}\label{alpha}
 \sum\limits_{v \in S}\alpha_v\le |S|+|N(S)|-1 \qquad\textnormal{for all~~} \emptyset \neq S\subseteq U, 
\end{equation} 
\begin{equation}\label{beta}
\sum\limits_{v \in S}\beta_v \geq \omega(G-S)+|S|-1 \qquad\textnormal{for all~~} S\subseteq U.
\end{equation}
\end{theorem}

The proof of Theorem \ref{1} in \cite{book} uses structural graph theory.
The question of whether a tree as desired in Theorem \ref{1} can be found in polynomial time is not asked and a cost function for the edges of  $G$ is not considered there. 

The remaining content of this article is organized as follows. The Only-If-Part of Theorem~\ref{1} can be seen easily and, for the sake of completeness,  it is given here. Using the theory of matroid intersection developed by Edmonds (\cite{e3},\cite{e2}), we present a new and short proof of the If-Part of Theorem \ref{1}. By this approach, we also can show the forthcoming Theorem \ref{2}.  

\begin{theorem}\label{2}
Given a  graph $G$, a real weight function $c:E(G)\rightarrow R$, an independent set $U\subseteq V(G)$, and sets $\{\alpha_v~|~v \in U\}$ and $\{\beta_v~|~v \in U\}$  of non-negative integers.  
In polynomial time,
we can decide whether $G$ contains a spanning tree $T$ such that $\alpha_v \le d_T (v) \le \beta_v$ for all $v \in  U$.
If such a tree exists, then we can find one of minimum cost in polynomial time.
\end{theorem}

\section{Proofs of  Theorem \ref{1} and Theorem \ref{2}}\label{strucalg}

This section is dedicated to proving  Theorem \ref{1} and Theorem \ref{2}. We review the necessary matroidal tools,  introduce the concrete  matroids that we need, and   proceed to the actual proofs of Theorems  \ref{1} and \ref{2}.

{\color{red}In this article, all matroids in consideration are finite.} For basics in matroid theory, see \cite[Chapter~5]{book}. A brief summary of the  definitions of a  matroid by independent sets or by bases and to the axioms of its rank function can be found in  \cite{wiki}. 
We need two results of Edmonds on the intersection of two matroids. The first one is of algorithmic nature and can be found in~\cite{e2}. Recall that a basis of a matroid $M$ is a maximum independent set of $M$.

\begin{theorem}\label{edmalg}
Let $M_1,M_2$ be two matroids on a common  weighted ground set $E$ with polynomial time independence oracles for
$M_1$ and $M_2$ being available.
Then a minimum-weight common basis $B$ of $M_1$ and $M_2$ (if exists) can be
found in  polynomial time.
\end{theorem}

The next one contains a structural characterization and can be found in \cite{e3}.

\begin{theorem}\label{edm}
Let $M_1,M_2$ be two matroids on a common  ground set  $E$ with rank functions  $r_1,r_2$, respectively, and let $\mu$ be a positive integer. Then there exists  $I \subseteq  E $ with $|I|=\mu$
such that $I$ is a common independent set of $M_1$ and  $M_2$ if and only if  $r_1(X)+r_2( E\setminus X)\geq \mu$ for all $ X \subseteq  E$.
\end{theorem}

Let $M$ be a  matroid on ground set $E$ and $\cal{I}$ be its set of independent sets.
 Recall that the rank function $r$ of  $M$ is a mapping from the power set of $E$ into the set  of non-negative integers with $r(X)=\max \{|I|~|~I\subseteq X, I\in \cal{I}  \}$ for $X\subseteq E$. 	Throughout this paper, we use the following properties of the rank function $r$:
	\begin{itemize}[topsep=0em, partopsep=0em, parsep=0em, itemsep=0em]
		\item subcardinality, i.e.~$r(X)\le |X|$ for all $X\subseteq E$,
		\item submodularity, i.e.~$r(X\cup Y)+r(X\cap Y)\le r(X)+r(Y)$ for all $X,Y\subseteq E$,
		\item monotonicity, i.e.~$r(X)\le r(Y)$ for all $X\subseteq Y \subseteq E$.
	\end{itemize}
 
Given  a stable set $U\subseteq V(G)$ of $G$ and sets $\{\alpha_v~|~v \in U\}$ and $\{\beta_v~|~v \in U\}$  of non-negative integers, we now describe two   matroids with common ground set $E(G)$ defined by bases and by independent sets, respectively,  that we need for the proof of Theorems \ref{1} and \ref{2}.  We acknowledge that similar ideas have been used by Edmonds when dealing with arborescence packings \cite{e4}.

Let $G$ be a graph. For $S \subseteq V(G)$, let  $\delta_G(S)$ denote the set of edges of $E(G)$ with exactly one
endvertex in $S$ and we  use $d_G(S)$ for $|\delta_G(S)|$, where a single element set $\{x\}$ is abbreviated
to $x$.  Further, if $X\subseteq E(G)$, then   $\omega_G(X)$ stands for the number of components of the graph with vertex set $V(G)$ and edge set $X$.

The \emph{first matroid $M_1$} is a \emph{generalized partition matroid}, where a set $B\subseteq E(G)$ is a basis of~$M_1$ if $\alpha_v \leq |B \cap \delta_{G}(v)|\leq \beta_v$ holds for all $v \in U$ and $|B|=|V(G)|-1$.
In \cite{hs}, the rank function $r_1$ of $M_1$  for  $X \subseteq E(G)$ is proven to be $$r_1(X)=\min\{A(X),B(X)\},$$ where
\[A(X)=|V(G)|-1-\sum\limits_{v \in U}\max\{\alpha_v-|X \cap \delta_{G}(v)|,0\}\]
and 
\[B(X)=\sum\limits_{v \in U}\min\{\beta_v,|X \cap \delta_{G}(v)|\}+|X\setminus \delta_G(U)|.\]

 The following Proposition \ref{welld} is needed to make sure that the   bases $B$  of $M_1$  fulfil the necessary properties (see  \cite{wiki}). It can be found in \cite{hs}.

\begin{proposition}\label{welld}
$M_1$ is well-defined if  and only if 
\medskip

\begin{enumerate}
\item $\alpha_v \leq\min\{\beta_v,d_G(v)\}$  for all $v \in U$ and
\item $\sum\limits_{v \in U}\alpha_v\leq |V(G)|-1\leq \sum\limits_{v \in U}\min\{\beta_v,d_G(v)\}+|E(G-U)|$.
\end{enumerate}  
\end{proposition}

The \emph{second matroid $M_2$} is the well-known \emph{cycle matroid} of $G$, where $I \subseteq E(G)$ is an independent set of $M_2$ if $I$ induces an acyclic subgraph of $G$.
The rank function $r_2$ of $M_2$ for  $X \subseteq E(G)$ is  
\[r_2(X)=|V(G)|-\omega_{G}(X).\] 

The forthcoming Lemma \ref{inter} is the link between spanning trees of bounded degree and matroid intersection.

\begin{lemma}\label{inter}
Let $G$ be a graph, $U$ be a stable set, and $\{\alpha_v~|~v \in U\}$ and $\{\beta_v~|~v \in U\}$ be two sets of non-negative integers.
An edge set $X\subseteq E(G)$ induces a spanning tree $T$ of $G$ that satisfies $\alpha_v\leq d_T(v)\leq \beta_v$ for all $v\in U$ if and only if
\begin{enumerate}[topsep=0em, partopsep=0em, parsep=0em, itemsep=0em]
\item $M_1$ is well-defined and
\item $X$ is a common independent set of $M_1$ and $M_2$ of size $|V(G)|-1$.
\end{enumerate}
\end{lemma}
\begin{proof}
The If-Part is easily seen as every common independent set $I$ of $M_1$ and $M_2$ induces an acyclic graph $F$ with $\alpha_v \leq |I \cap \delta_{F}(v)|\leq \beta_v$ for all $v \in U$. 
Clearly,  as $|X|=|V(G)|-1$, we obtain that $F$ is a tree.

For the Only-If-Part, we assume that $X\subseteq E(G)$ induces a spanning tree $T$ with \linebreak $\alpha_v \leq |I \cap \delta_{T}(v)|\leq \beta_v$ for all $v \in U$.
We have $\alpha_v \leq d_T(v)\leq \min\{\beta_v,d_G(v)\}$ for all $v \in U$ and 
\[\sum\limits_{v \in U}\alpha_v\leq \sum\limits_{v \in U}d_T(v)\leq |V(G)|-1= \sum\limits_{v \in U}d_T(v)+|E(T-U)|\leq \sum\limits_{v \in U}\min\{\beta_v,d_G(v)\}+|E(G-U)|.\]
By Proposition \ref{welld}, $M_1$ is well-defined. 
Furthermore, it is easily seen that $X$ is an independent set of $M_1$ and $M_2$.
\end{proof}

Now we are ready to proceed to the proofs of Theorem \ref{1} and \ref{2}.

\begin{proposition}\label{equiv}
Condition 
\[\sum\limits_{v \in S}\alpha_v\le |S|+|N(S)|-1 \qquad\textnormal{for all~~} \emptyset \neq S\subseteq U\tag{\ref{alpha}}\]
of Theorem {\em \ref{1}} is equivalent to 
\begin{equation}\label{alpha'}
 \sum\limits_{v \in S}\alpha_v\le |V(G)|-\omega_G(\delta_G(S)) \qquad\textnormal{for all~~} \emptyset \neq S\subseteq U. 
\end{equation}
\end{proposition}

\begin{claimproof}
For $\emptyset \neq S\subseteq U$, let  $H_S$ be the graph with vertex set $S\cup N(S)$ and edge set $\delta_G(S)$. Since $|V(G)|-\omega_G(\delta_G(S))=|S|+|N(S)|-\omega(H_S)$, it follows that (\ref{alpha'}) is equivalent to 

\begin{equation}\label{alpha''}
 \sum\limits_{v \in S}\alpha_v\le |S|+|N(S)|-\omega(H_S)~ \qquad\textnormal{for all~~} \emptyset \neq S\subseteq U. 
\end{equation}
Clearly,  (\ref{alpha})  follows from (\ref{alpha''}) since $S\neq \emptyset$ implies $\omega(H_S)\ge 1$.
Conversely, let $S=\bigcup\limits_{i=1}^{\omega(H_S)} S_i$ be the unique decomposition of $S$ into mutually disjoint and non-empty sets $S_i$, where the sets $S_i\cup N(S_i)$ for $i=1,\ldots,\omega(H_S)$ induce the components of $H_S$.   Summing up  (\ref{alpha})  for all $S_i$ implies (\ref{alpha''}) for $S$.
\end{claimproof}

\begin{proof}[Proof of the Only-If-Part of Theorem \ref{1}:] \

Let us assume that $G$ contains a spanning tree $T$ such that $\alpha_v\leq d_T(v)\leq \beta_v$  for all $v \in U$. We have to show 
 that (\ref{alpha}) and (\ref{beta}) hold. 
 We will use the equality $|X|=\omega_G(E(T)\setminus X)-1$ for $X\subseteq E(T).$ Note that this fact 
 can be seen easily by induction on $|X|$. 

Let $S \subseteq U$ and  $X=E(T)\setminus  \delta_G(S)$. It follows
\begin{align*}
	\sum\limits_{v \in S}\alpha_v&\le \sum\limits_{v \in S}d_T(v)=|E(T)\cap  \delta_G(S)|=|V(G)|-1-|E(T)\setminus \delta_G(S)|\\
		&=|V(G)|-\omega_G(E(T) \cap \delta_G(S))\leq |V(G)|-\omega_G(\delta_G(S)).
\end{align*}
Thus, (\ref{alpha}) holds by Proposition \ref{equiv}.

Let $S \subseteq U$ and  $X=E(T)\cap \delta_G(S)$. We obtain
\begin{align*}
	\omega(G-S)+|S|-1&\leq \omega(T-S)+|S|-1 =\omega_G(E(T)\setminus \delta_G(S))-1 \\
		&=|E(T)\cap \delta_G(S)|= \sum\limits_{v \in S}d_T(v) \leq \sum\limits_{v \in S}\beta_v,
\end{align*}
Hence, (\ref{beta}) follows and our proof for the Only-If-Part is complete.
\end{proof}

\begin{proof}[Proof of the If-Part of Theorem {\ref{1}}:] \ 

Within this proof, we use the generalized partition matroid~$M_1$ and the cycle matroid~$M_2$. First of all, we verify well-definedness of $M_1$.

\medskip

\begin{proposition}\label{rechnung}
If Conditions {\em (\ref{alpha})} and {\em (\ref{beta})} of Theorem \ref{1} hold, then  $M_1$ is well-defined.
\end{proposition}
\begin{claimproof}
Taken $S=\{v\}$ for $v\in U$, we have $\alpha_v \le d_G(v)$ by (\ref{alpha}), and so $\alpha_v\le \min\{\beta_v,d_G(v)\}$. Furthermore, taken $S=U$, it follows
$\sum\limits_{v \in U}\alpha_v \leq |V(G)|-1$ from (\ref{alpha}).

As easily seen by induction on $\omega(G-S)$, we have 
\begin{equation}\label{rechnung:1}
|V(G)\setminus S|-|E(G-S)|-\omega(G-S)\leq 0
\end{equation}
for all graphs $G$ and for all $S\subseteq V(G)$.

Let $S=\{v\in U~|~d_G(v)\ge \beta_v\}$. Note that $S$ is possibly empty. As $U$ is stable, we have 
\begin{equation}\label{rechnung:2}
|E(G-S)|=\sum\limits_{v \in U\setminus S} d_G(v)+|E(G-U)|
\end{equation}
since $U$ is a stable set.
Moreover, 
\begin{equation}\label{rechnung:3}
\omega(G-S)+|S|-1\leq \sum\limits_{v \in S}\beta_v
\end{equation} by Condition (\ref{beta}).
Summing \eqref{rechnung:1}, \eqref{rechnung:2}, and \eqref{rechnung:3} leads to
\[|V(G)\setminus S|+|S|-1\leq \sum\limits_{v \in S} \beta_v + \sum\limits_{v \in U\setminus S} d_G(v)+|E(G-U)|.\] 
From 
\[|V(G)\setminus S|+|S|=|V(G)|\quad\textnormal{and}\quad \sum\limits\limits_{v \in S} \beta_v + \sum\limits_{v \in U\setminus S} d_G(v)=\sum\limits_{v \in U} \min\{\beta_v, d_G(v)\},\]
it further follows
\[|V(G)|-1\leq \sum\limits_{v \in U}\min\{\beta_v,d_G(v)\}+|E(G-U)|.\] Hence, 
$M_1$ is well-defined by Proposition \ref{welld}.\end{claimproof}

Having shown that $M_1$ and $M_2$ are well-defined, we are now ready to use matroid intersection. 
In order to apply Theorem~\ref{edm}, we use the forthcoming Proposition~\ref{r2}. It can be easily seen by using submodularity and subcardinality of $r_2$, and by the fact
\[E(G)\setminus X'=(E(G)\setminus X)\cup (X\setminus X').\]

\begin{proposition}\label{r2}
 If $X'\subseteq X\subseteq E(G)$, then $r_2(E(G)\setminus X')\le r_2(E(G)\setminus X)+ |X\setminus X'|$.
\end{proposition} 

We now assume that Conditions (\ref{alpha}) and (\ref{beta}) of Theorem \ref{1} hold and suppose, 
for the sake of a contradiction, that  $G$ does not contain a spanning tree $T$ with $\alpha_v\leq d_T(v)\leq \beta_v$ for all $v \in U$.
By Lemma \ref{inter}, there is no $I \subseteq E(G)$ with $|I|=|V(G)|-1$ being independent in both $M_1$  and $M_2$. By Theorem \ref{edm}, there are sets $X \subseteq E(G)$ such that 
\[r_1(X)+r_2(E(G)\setminus X)\leq |V(G)|-2.\]
We note that $r_1(X)=\min\{A(X),B(X)\}$, which leads to the following case distinction. 

\medskip

\begin{case}\label{case1}
There is a set $X\subseteq E(G)$ such that $A(X)+r_2(E(G)\setminus X) \leq |V(G)|-2.$
\end{case}
Let us choose a set $X\subseteq E(G)$ such that
\begin{enumerate}[topsep=0.2em, partopsep=0em, parsep=0em, itemsep=0em] 
\item[(a)] $A(X)+r_2(E(G)\setminus X)\leq |V(G)|-2$ and
\item[(b)] with respect to (a), $\{v~|~0<|X\cap \delta_{G}(v)|<\alpha_v, v\in U\}$ is of minimum size.
\end{enumerate}
Recall that 
 \[A(X)=|V(G)|-1-\sum\limits_{v \in U}\max\{\alpha_v-|X\cap \delta_{G}(v)|,0\}.\]

\begin{claim}\label{c1}
$\{v: 0<|X\cap \delta_{G}(v)|<\alpha_v, v\in U\}=\emptyset$

\end{claim}

\begin{claimproof}
For the sake of a contradiction, we suppose that  there is a vertex $v_0 \in U$ such that $0<|X\cap \delta_{G}(v_0)|< \alpha_{v_0}.$

Let $X'=X\setminus \delta_{G}(v_0)$. It follows $X'\cap \delta_{G}(v_0)= \emptyset$ and 
 $X'\cap \delta_{G}(v)=X\cap \delta_{G}(v)$ for all $v\in U\setminus \{v_0\}$ since $U$ is a stable set. Hence,
\begin{align*} 
A(X)-A(X') &= \sum\limits_{v \in U}\max\left\{\alpha_v- |X'\cap \delta_{G}(v)|,0\right\}-\sum\limits_{v \in U}\max\left\{\alpha_v- |X\cap \delta_{G}(v)|,0\right\}\\
&=\alpha_{v_0}-(\alpha_{v_0}-|X\cap \delta_{G}(v_0)|)
               = |X\cap \delta_{G}(v_0)|=|X\setminus X'|.
\end{align*}
 By Proposition \ref{r2}, 
\[A(X')+r_2(E(G)\setminus X')\leq A(X)+r_2(E(G)\setminus X)\leq |V(G)|-2.\]
However, this fact contradicts the choice of $X$ since $|X'\cap \delta_G(v_0)|=0$ and
\[\{v~|~0<|X'\cap \delta_{G}(v)|<\alpha_v\}\cup\{v_0\}=\{v~|~0<|X\cap \delta_{G}(v)|<\alpha_v\}.\\[-2.2em]\]
\end{claimproof}

Let  $S=\{v \in U~|~X\cap \delta_{G}(v)=\emptyset\}$. Note that $S$ is possibly empty. We have
\[A(X)=|V(G)|-1-\sum\limits_{v \in S}\alpha_v\] by Claim \ref{c1}.
Using
$\delta_G(S)\subseteq E(G)\setminus X$ and the monotonicity of $r_2$, it follows
\begin{align*}
     \sum_{v \in S}\alpha_v &= |V(G)|-1-A(X) \geq [A(X)+r_2(E(G)\setminus X)+1] -A(X)\\
     &\geq r_2(\delta_G(S))+1=|V(G)|-\omega_{G}(\delta_G(S))+1,
\end{align*}
which is impossible if  $S= \emptyset$. In case $S\neq \emptyset$, this fact together with Proposition \ref{equiv} contradicts Condition (\ref{alpha}) of Theorem \ref{1}, and so the proof of Case~\ref{case1} is complete.

\medskip

\begin{case}\label{case2}
There is a set $X\subseteq E(G)$ such that $B(X)+r_2(E(G)\setminus X) \leq |V(G)|-2.$
\end{case}

Let us choose a set $X\subseteq E(G)$ such that
\begin{enumerate}[topsep=0.2em, partopsep=0em, parsep=0em, itemsep=0em] 
\item[(a)] $B(X)+r_2(E(G)\setminus X)\leq |V(G)|-2$,
\item[(b)] with respect to (a), $X\setminus \delta_G(U)$ is of minimum size, and
\item[(c)] with respect to (a) and (b), $\{v: 0<|X\cap \delta_{G}(v)|<\beta_v, v\in U\}$ is of minimum size.
\end{enumerate}
Recall that 
 \[B(X)=\sum\limits_{v \in U}\min\{\beta_v,|X\cap \delta_G(v)|\}+|X\setminus \delta_G(U)|.\]

\begin{claim}\label{c2}
$X\subseteq \delta_G(U)$
\end{claim}%
\begin{claimproof}
For the sake of a contradiction, we suppose $X\setminus \delta_G(U)\neq \emptyset$. Let $X'=X \cap \delta_G(U)$. 

Hence, $|X'\setminus \delta_G(U)|=0 $, $X'\cap \delta_G(U)=X'$,  $X\setminus X'=X\setminus \delta_G(U)$, and
$B(X)-B(X')
=|X\setminus X'|$.
By Proposition \ref{r2}, we have 
\[B(X')+r_2(E(G)\setminus X')
\leq 
B(X)+r_2(E(G)\setminus X)\leq |V(G)|-2.\] 
However, this fact contradicts the choice of $X$ since $|X'\setminus \delta_G(U)|=0<|X\setminus \delta_G(U)|$.
\end{claimproof}

\begin{claim}\label{c3}
$\{v: 0<|X\cap \delta_{G}(v)|<\beta_v, v\in U\}=\emptyset$
\end{claim}
\begin{claimproof}
For the sake of a contradiction, we suppose that there is a vertex $v_0 \in U$ such that $0<|X\cap \delta_{G}(v_0)|<\beta_{v_0}$. 
Let $X'=X\setminus \delta_{G}(v_0)$.
Since 
$X'\subseteq X\subseteq \delta_G(U)$, it follows
\begin{align*} 
B(X)-B(X') &= \sum\limits_{v \in U}\min\left\{\beta_v,|X\cap \delta_{G}(v)|\right\}-\sum\limits_{v \in U}\min\left\{\beta_v,|X'\cap \delta_{G}(v)|\right\}\\
               &= |X\cap \delta_{G}(v_0)|-|X'\cap \delta_{G}(v_0)|=|X\setminus X'|.
\end{align*}
By Proposition \ref{r2}, 
\[B(X')+r_2(E(G)\setminus X')\leq B(X)+r_2(E(G)\setminus X)\leq |V(G)|-2.\]
However, this fact together with the fact $X'\setminus \delta_G(U)=\emptyset$ contradicts the choice of $X$ since
$0=|X'\cap \delta_{G}(v_0)|$
and
\[\{v: 0<|X'\cap \delta_{G}(v)|<\beta_v, v\in U\}\cup\{v_0\}=\{v: 0<|X\cap \delta_{G}(v)|<\beta_v, v\in U\}.\\[-2.2em]\]
\end{claimproof}

Let  $S=\{v \in U~|~|X\cap \delta_{G}(v)|\ge \beta_v \}$. Note that $S$ is possibly empty. We have $X\cap \delta_{G}(v)=\emptyset$ for all $v\in U\setminus S$ by Claim~\ref{c3}.
Therefore, since $X\subseteq \delta_{G}(U)$ by the choice of $X$, we further have $X\subseteq \delta_{G}(S)$.
Additionally, Claim \ref{c3} now implies $B(X)=\sum\limits_{v \in S}\beta_v$. We obtain
\begin{align*}
	\sum\limits_{v \in S}\beta_v=B(X)&\leq |V(G)|-2-r_2(E(G)\setminus X)=\omega_G(E(G)\setminus X)-2\leq \omega_G(E(G)\setminus \delta_G(S))-2 \\ &= \omega(G-S)+|S|-2,
\end{align*}
a contradiction to Condition (\ref{beta}) of Theorem \ref{1}, and so the proof of Case~\ref{case2} is complete.

As Case 1 and Case 2 cover every possible situation, our proof for the If-Part is complete.
\end{proof}

\begin{proof}[Proof of Theorem {\em \ref{2}}]
By Lemma~\ref{inter}, it suffices to check if 
\begin{enumerate}[topsep=0em, partopsep=0em, parsep=0em, itemsep=0em]
\item $M_1$ is well-defined and
\item $M_1$ and $M_2$ have a common independent set of size $|V(G)|-1$.
\end{enumerate}
Using Proposition~\ref{welld}, (i) can be checked in linear time, and we can draw our attention to checking~(ii). Note that a basis of $M_2$ is an edge set of a spanning tree of $G$, and thus is of size $|V(G)|-1$. Hence, Theorem~\ref{edmalg} completes the proof.
\end{proof}

 \section*{Declarations}
No data has been used. There are no competing interests. All authors contributed equally.

\end{document}